\documentclass[11pt]{amsart}

\usepackage[T1]{fontenc}
\usepackage[british]{babel}
\usepackage{lmodern}
\usepackage{amsmath,amssymb,amsthm,mathtools}
\usepackage{enumitem}
\usepackage{microtype}
\usepackage[hidelinks]{hyperref}

\makeatletter
\@namedef{subjclassname@2020}{\textup{2020} Mathematics Subject Classification}
\makeatother

\hypersetup{
  pdftitle={Kirszbraun extensions preserving uniform distance in Hilbert spaces},
  pdfauthor={Krzysztof J. Ciosmak},
  pdfkeywords={Lipschitz extension, Kirszbraun theorem, Hilbert space,
    barycentric inequality, minimal support, weak transport,
    convex Poincare inequality}
}

\newtheorem{theorem}{Theorem}[section]
\newtheorem{proposition}[theorem]{Proposition}
\newtheorem{corollary}[theorem]{Corollary}
\newtheorem{lemma}[theorem]{Lemma}
\newtheorem{example}[theorem]{Example}
\theoremstyle{remark}
\newtheorem{remark}[theorem]{Remark}

\newcommand{\R}{\mathbb{R}}
\newcommand{\conv}{\operatorname{conv}}
\newcommand{\dist}{\operatorname{dist}}
\newcommand{\ip}[2]{\left\langle #1,#2\right\rangle}
\newcommand{\norm}[1]{\left\lVert #1\right\rVert}

\title[Kirszbraun extensions preserving uniform distance]
{Kirszbraun extensions preserving uniform distance in Hilbert spaces}

\author{Krzysztof J. Ciosmak \href{https://orcid.org/0000-0001-9571-1160}{\includegraphics{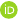}}}
\address{Beijing Institute of Mathematical Sciences and Applications,
No.~544, Hefangkou Village, Huaibei Town, Huairou District,
Beijing 101408, China}
\email{kciosmak@bimsa.cn}

\begin{document}

\begin{abstract}
Let $X$ be a subset of a real Hilbert space and let
$v\colon X\to Y$, where $Y$ is a real Hilbert space.  We prove that the following
conditions are equivalent: whenever $A\subset X$, $\rho\geq0$, and
$u\colon A\to Y$ is $1$-Lipschitz with
$\norm{u(x)-v(x)}\leq\rho$ for $x\in A$, there is a $1$-Lipschitz extension
$\widetilde u\colon X\to Y$ with
$\norm{\widetilde u(x)-v(x)}\leq\rho$ for $x\in X$; and for every
$1\leq k\leq\dim Y$,
\[
 \norm{v(x_0)-\sum_{i=1}^k t_i v(x_i)}
 \leq
 \norm{x_0-\sum_{i=1}^k t_i x_i}
\]
whenever $x_0,\ldots,x_k\in X$, $t_1,\ldots,t_k\geq0$, and
$\sum_{i=1}^k t_i=1$.  Previous necessity results required
$\dim Y\leq3$ or convexity of $X$.  For finite-dimensional targets, an
application gives an exact data processing characterisation for a
finite branching hierarchy connecting Wasserstein and barycentric weak
transport.  If $Y$ is infinite-dimensional or
$\dim\operatorname{Aff}X+1\leq\dim Y$, we also obtain a lifting theorem for
convex Lipschitz functions and transfer convex Poincar\'e inequalities without
increasing the constant.
\end{abstract}

\maketitle

\section{Introduction and main result}

The Kirszbraun--Valentine theorem states that a $1$-Lipschitz map between
subsets of real Hilbert spaces admits a $1$-Lipschitz extension to the whole
domain space \cite{Kirszbraun1934,Valentine1945}.  Kirszbraun extension
theorems for metric spaces with compatible upper and lower Alexandrov
curvature bounds were proved by Lang and Schroeder
\cite{LangSchroeder1997}.  A proof based on convex analysis and an explicit
extension formula appear, respectively, in \cite{ReichSimons2005} and
\cite{AzagraLeGruyerMudarra2021}.
Beyond existence, continuous selection in the uniform norm has been studied for
extensions that preserve the Lipschitz constant of bounded maps, in
finite-dimensional and Hilbert space settings
\cite{Kopecka2012Continuous,Kopecka2012Bootstrapping}.  For Hilbert spaces,
the extension may moreover be selected with its image contained in the closed
convex hull of the original image \cite{KopeckaReich2011}.  The continuous
selection theorem was extended to bounded nonexpansive maps from subsets of
spaces with Alexandrov curvature bounded below by $\kappa$ into complete
spaces with curvature bounded above by the same $\kappa$, for $\kappa\leq0$.
When $\kappa=0$, the result also covers bounded Lipschitz maps, preserving both
the Lipschitz constant and the requirement that the image lie in the closed
convex hull of the original image
\cite{EspinolaNicolae2015}.  Separately,
Kirszbraun--Valentine extensions preserving invariance under groups of
isometries of Hilbert spaces were obtained in
\cite{CavagnariSavareSodini2025}.

Here the extension is required to satisfy a second simultaneous constraint.
Given a reference map
$v\colon X\to Y$ and a $1$-Lipschitz map $u\colon A\to Y$, $A\subset X$, one asks for a
$1$-Lipschitz extension that does not increase the uniform distance from $v$.
This question has been already addressed in the study of continuity properties of Lipschitz
extensions \cite{Ciosmak2021,Ciosmak2024}. In \cite[Theorem~1.2]{Ciosmak2024} the sufficiency of the barycentric condition $(\mathrm{C})$ below has been established for $(\mathrm{E})$ to hold true.
For finite-dimensional targets, the necessity of the barycentric condition
below was conjectured in \cite[\S4.3.5, p.~88]{CiosmakThesis}.  The conjectured
implication was subsequently proved when $\dim Y\leq3$ or when $X$ is convex
\cite[Theorem~1.2]{Ciosmak2024}.  Our main theorem proves it without either
restriction and extends the resulting characterisation to arbitrary real
Hilbert targets.

The strategy of the proof is the same as in \cite{Ciosmak2024} for $\mathrm{dim}Y\in\{1,2,3\}$. It relies on constructing an isometric embedding $\{u(x_1),\dotsc,u(x_k)\}$ of $\{x_1,\dotsc,x_k\}$ in $Y$ in such a way that $\norm{v(x_i)-u(x_i)}$ for $i=1,2\dotsc,k$ are suitably controlled and $v(x_1)-u(x_i)$ is moved along the direction orthogonal to the simplex $\{v(x_1),\dotsc,v(x_k)\}$. 
The obstruction to the earlier proof was an isometric embedding step.
Example~4.10 of \cite{Ciosmak2024} shows that the required isometric embedding need not
exist for an arbitrary four-point configuration.  The point of the present
argument is that arbitrary configurations are unnecessary. Proposition~\ref{prop:minimal-simplex},
shows that every strict violation of $(\mathrm{C})$ admits an affinely independent simplex
witness and that a suitable scalar component of $v$ has a $1$-Lipschitz affine
interpolant on that support.  Those two facts give exactly the embedding needed
in the limiting argument.  Thus the four-point example obstructs the former proof
method, not the conjectured characterisation.

\begin{theorem}\label{thm:main}
Let $Z$ and $Y$ be real Hilbert spaces, let $X\subset Z$, and let $v\colon X\to Y$.
The following are equivalent:
\begin{enumerate}[label={},leftmargin=2.6em,labelsep=0.5em,itemsep=0.8em]
\item[\textup{$(\mathrm E)$}] for every $A\subset X$, every $\rho\geq0$, and
every $1$-Lipschitz map $u\colon A\to Y$ satisfying
\[
  \norm{u(x)-v(x)}\leq\rho \qquad\text{for }x\in A,
\]
there is a $1$-Lipschitz extension $\widetilde u\colon X\to Y$ satisfying
\[
  \norm{\widetilde u(x)-v(x)}\leq\rho \qquad\text{for }x\in X;
\]
\item[\textup{$(\mathrm C)$}] for every positive integer $k\leq\dim Y$,
one has
\[
  \norm{v(x_0)-\sum_{i=1}^{k}t_i v(x_i)}
  \leq
  \norm{x_0-\sum_{i=1}^{k}t_i x_i}
\]
for all $x_0,x_1,\ldots,x_k\in X$ and all $t_1,\ldots,t_k\geq0$ with
$\sum_{i=1}^k t_i=1$.
\end{enumerate}
\end{theorem}

When $A$ is nonempty and the deviation is bounded, taking
\[
 \rho=\sup\{\norm{u(x)-v(x)}\mid x\in A\}
\]
in $(\mathrm E)$ shows that the uniform distance is preserved exactly,
since the extension agrees with $u$ on $A$.

Consequently, conditions \emph{(i)}--\emph{(iii)} of
Theorem~1.2 in \cite{Ciosmak2024} are equivalent for arbitrary real Hilbert
targets and arbitrary subsets $X$.  For finite-dimensional targets, this
resolves the conjecture stated in \cite[\S4.3.5, p.~88]{CiosmakThesis}.

In addition, we introduce a finite branching weak transport cost and show that, for finite-dimensional targets, property
$(\mathrm E)$ is exactly its data processing condition, that is, contraction of the finite-branching cost under pushforward by \(v\), at branching level
$\dim Y$. Under a dimension assumption that yields the unrestricted finite
barycentric condition, we also obtain convex Lipschitz lifting and transfer of
convex Poincar\'e inequalities without deterioration of the constant.

\section{A minimal support simplex principle}

The following proposition is independent of the extension problem and is valid
in an arbitrary real normed space.  It shows that every strict scalar violation of the
barycentric inequality has a simplex witness with a $1$-Lipschitz affine
interpolant.

\begin{proposition}[Minimal support simplex principle]
\label{prop:minimal-simplex}
Let $Z$ be a real normed space, let $X\subset Z$, and let $f\colon X\to\R$.
Suppose that, for some $x_0,x_1,\ldots,x_n\in X$ and
$t_1,\ldots,t_n\geq0$ with $\sum_{i=1}^n t_i=1$,
\begin{equation}\label{eq:scalar-violation}
 f(x_0)-\sum_{i=1}^n t_i f(x_i)
 >
 \norm{x_0-\sum_{i=1}^n t_i x_i}.
\end{equation}
Then there exist distinct points
$y_1,\ldots,y_m\in\{x_1,\ldots,x_n\}\setminus\{x_0\}$ and numbers
$\lambda_1,\ldots,\lambda_m>0$ with $\sum_{i=1}^m\lambda_i=1$, such that
\begin{enumerate}[label=\textup{(\roman*)}]
\item $m\leq n$ and $y_1,\ldots,y_m$ are affinely independent;
\item
\[
 f(x_0)-\sum_{i=1}^m\lambda_i f(y_i)
 >
 \norm{x_0-\sum_{i=1}^m\lambda_i y_i};
\]
\item if
\[
 L=\operatorname{span}\{y_i-y_1\mid 1\leq i\leq m\},
\]
then the unique affine function $a$ on $y_1+L$ satisfying
$a(y_i)=f(y_i)$ for $1\leq i\leq m$ has the form
\[
 a(y_1+h)=f(y_1)+\ell(h)\qquad\text{for }h\in L,
\]
where $\ell\in L^*$ and $\norm{\ell}_{L^*}\leq1$.
\end{enumerate}
In particular,
\[
 m\leq\dim\operatorname{Aff}\{x_1,\ldots,x_n\}+1.
\]
\end{proposition}

Here and below, the support of a convex combination $\sum_{i=1}^mc_iz_i$ is the support of the measure $\sum_{i=1}^mc_i\delta_{z_i}$.  

\begin{proof}
First combine coefficients belonging to repeated points.  If $x_0$ occurs in
the resulting support with total weight $\tau$, then $\tau<1$, since otherwise
both sides of \eqref{eq:scalar-violation} vanish.  After deleting $x_0$ and
renormalising the remaining coefficients by $1-\tau$, both sides of the
strict inequality are divided by $1-\tau$.  Thus there remains a strict
violation whose support does not contain $x_0$.

Among all probability measures on the resulting finite set that give a strict
violation with the fixed point $x_0$, choose one whose support has
minimal cardinality $m$.  Write the corresponding violation as
\[
 f(x_0)-\sum_{i=1}^m\lambda_i f(y_i)>\norm{x_0-p},
 \qquad p=\sum_{i=1}^m\lambda_i y_i,
\]
where $y_1,\ldots,y_m$ are distinct and every $\lambda_i$ is positive.

Suppose first that $y_1,\ldots,y_m$ are affinely dependent.  Consider the
compact polytope
\[
 P=\left\{s\in\R^m\mid s_i\geq0\text{ for }1\leq i\leq m,\quad
 \sum_{i=1}^m s_i=1,\quad
 \sum_{i=1}^m s_i y_i=p\right\}.
\]
Minimise $s\mapsto\sum_{i=1}^m s_i f(y_i)$ over $P$, and choose an extreme point
$s$ of the minimising face.  Then $s$ is an extreme point of $P$.  The support of its convex representation is affinely independent.  Indeed, otherwise there would be a
nonzero vector $\alpha\in\R^m$, supported on the positive coordinates of $s$,
such that
\[
 \sum_{i=1}^m\alpha_i=0,
 \qquad \sum_{i=1}^m\alpha_i y_i=0.
\]
Thus, for all sufficiently small $r>0$, the distinct points $s+r\alpha$ and
$s-r\alpha$ would then belong to $P$ and have midpoint $s$. That would contradict the  fact that $s$ is an extreme point of $P$. Since
$y_1,\ldots,y_m$ are affinely dependent, the support of the convex representation of $s$
has fewer than $m$ elements.  Moreover,
\[
 f(x_0)-\sum_{i=1}^m s_i f(y_i)
 \geq f(x_0)-\sum_{i=1}^m\lambda_i f(y_i)
 >\norm{x_0-p}
 =\norm{x_0-\sum_{i=1}^m s_i y_i},
\]
contradicting the minimality of $m$.  Thus the support is affinely
independent.

Set $K=\conv\{y_1,\ldots,y_m\}$ and let $a$ be the affine interpolant of the
values $f(y_i)$.  For $z\in K$ define
\[
 G(z)=a(z)+\norm{x_0-z}.
\]
Since all $\lambda_i$ are positive, $p$ lies in the interior of $K$ relative
to its affine hull, and the chosen violation says $G(p)<f(x_0)$.  Since $K$
is compact and $G$ is continuous, we can choose a
minimiser $p_0\in K$.  If $p_0$ lies on the boundary of $K$ relative to its
affine hull, its unique barycentric representation uses fewer than $m$
vertices of $K$.  Since $G(p_0)\leq G(p)<f(x_0)$, this representation would again
give a strict violation with smaller support.  Consequently, $p_0$ lies in
the interior of $K$ relative to its affine hull.

Write
\[
 a(y_1+h)=f(y_1)+\ell(h)\qquad\text{for }h\in L.
\]
For any $h\in L$ and all sufficiently small $t>0$, both $p_0+th$ and
$p_0-th$ belong to $K$.  Minimality of $G(p_0)$ and the triangle inequality give
\[
 0\leq t\ell(h)+\norm{x_0-p_*-th}-\norm{x_0-p_*}
 \leq t\bigl(\ell(h)+\norm{h}\bigr).
\]
Hence $\ell(h)\geq-\norm{h}$.  Applying the same argument to $-h$ gives
$\ell(h)\leq\norm{h}$.  Therefore
$|\ell(h)|\leq\norm{h}$ for every $h\in L$, as required.
\end{proof}

For the vector-valued problem, the norming direction of an initial violation
calibrates the scalar reduction.

\begin{corollary}[Calibrated simplex witness]
\label{cor:calibrated-simplex}
Let $Z$ and $Y$ be real Hilbert spaces, let $X\subset Z$, and let $v\colon X\to Y$.
Suppose that $x_0,x_1,\ldots,x_n\in X$ and that
$t_1,\ldots,t_n\geq0$,
$\sum_{i=1}^n t_i=1$, satisfy
\[
 \norm{v(x_0)-\sum_{i=1}^n t_i v(x_i)}
 >\norm{x_0-\sum_{i=1}^n t_i x_i}.
\]
Then there exist $m\leq n$, affinely independent points
$y_1,\ldots,y_m$ in $\{x_1,\ldots,x_n\}$, none equal to $x_0$, positive
coefficients $\lambda_1,\ldots,\lambda_m$ with
$\sum_{i=1}^m\lambda_i=1$, and a unit vector $w\in Y$ such that
\begin{equation}\label{eq:calibrated-violation}
 \ip{w}{v(x_0)-\sum_{i=1}^m\lambda_i v(y_i)}
 >\norm{x_0-\sum_{i=1}^m\lambda_i y_i}.
\end{equation}
Moreover, if
\[
 L=\operatorname{span}\{y_i-y_1\mid 1\leq i\leq m\},
\]
then there is $g\in L$, $\norm{g}\leq1$, such that
\begin{equation}\label{eq:calibrated-gradient}
 \ip{w}{v(y_i)}
 =\ip{w}{v(y_1)}+\ip{g}{y_i-y_1}
 \qquad\text{for }1\leq i\leq m.
\end{equation}
 In particular, the same convex combination gives a strict violation of the barycentric inequality.
\end{corollary}

\begin{proof}
Put
\[
 \Delta=v(x_0)-\sum_{i=1}^n t_i v(x_i).
\]
The assumed strict inequality gives $\norm{\Delta}>0$.  Set
$w=\Delta/\norm{\Delta}$ and apply Proposition~\ref{prop:minimal-simplex} to
$f(\cdot)=\ip{w}{v(\cdot)}$.  The Riesz representation theorem identifies the
linear part of the resulting affine interpolant with a vector $g\in L$ with
$\norm{g}\leq1$.  Finally,
\[
 \norm{v(x_0)-\sum_{i=1}^m\lambda_i v(y_i)}
 \geq\ip{w}{v(x_0)-\sum_{i=1}^m\lambda_i v(y_i)},
\]
so \eqref{eq:calibrated-violation} also gives a violation of the barycentric inequality.
\end{proof}

\section{The isometric test inequality}

We next record the consequence of $(\mathrm E)$ that converts an isometric
embedding into a barycentric estimate.

\begin{lemma}\label{lem:isometric-test}
Let $Z$ and $Y$ be real Hilbert spaces, let $X\subset Z$, and let $v\colon X\to Y$. Assume that $(\mathrm{E})$ is satisfied for $v$.  Let $x_0,x_1,\ldots,x_k\in X$, let
$t_1,\ldots,t_k\geq0$ with $\sum_{i=1}^k t_i=1$, and suppose that
$y_1,\ldots,y_k\in Y$ form an isometric copy
of $x_1,\ldots,x_k$:
\[
  \norm{y_i-y_j}=\norm{x_i-x_j}
  \qquad\text{for }1\leq i,j\leq k.
\]
Set
\[
  \bar x=\sum_{i=1}^k t_i x_i,\qquad
  \bar y=\sum_{i=1}^k t_i y_i,\qquad
  \rho=\max\{\norm{y_i-v(x_i)}\colon 1\leq i\leq k\}.
\]
Then
\begin{equation}\label{eq:isometric-test}
  \norm{\bar y-v(x_0)}
  \leq \rho+\norm{x_0-\bar x}.
\end{equation}
\end{lemma}

\begin{proof}
The distance identities ensure that the assignment $u(x_i)=y_i$ defines an
isometric embedding of $A=\{x_1,\ldots,x_k\}$ into $Y$.  Apply property $(\mathrm E)$ of $v$ with the subset $A$, the map $u$, and the bound $\rho$. Write $z=\widetilde u(x_0)$.  Thus
$\norm{z-v(x_0)}\leq\rho$.  Summing
$\norm{z-y_i}^2\leq\norm{x_0-x_i}^2$ for $1\leq i\leq k$, with weights $t_i$,
and using the two
variance identities
\[
 \sum_{i=1}^k t_i\norm{z-y_i}^2
 =\norm{z-\bar y}^2+
   \frac12\sum_{i=1}^k\sum_{j=1}^k t_i t_j\norm{y_i-y_j}^2
\]
and
\[
 \sum_{i=1}^k t_i\norm{x_0-x_i}^2
 =\norm{x_0-\bar x}^2+
   \frac12\sum_{i=1}^k\sum_{j=1}^k t_i t_j\norm{x_i-x_j}^2
\]
gives $\norm{z-\bar y}\leq\norm{x_0-\bar x}$, since the double sums agree.
The triangle inequality now yields \eqref{eq:isometric-test}.
\end{proof}

\section{Proof of the characterisation}

\begin{proof}[Proof of Theorem~\ref{thm:main}]
Only that $(\mathrm E)$ implies $(\mathrm C)$ remains to be proved.  Suppose that
$(\mathrm E)$ holds but that $(\mathrm C)$ fails.  The inequality in
$(\mathrm C)$ fails for some finite $n\leq\dim Y$.  By
Corollary~\ref{cor:calibrated-simplex}, after relabeling, there are affinely
independent $x_1,\ldots,x_k$, positive weights $t_1,\ldots,t_k$ with
$\sum_{i=1}^k t_i=1$, a
unit vector $w\in Y$, and a vector
\[
 g\in L:=\operatorname{span}\{x_i-x_1\mid 1\leq i\leq k\},
 \qquad \norm{g}\leq1,
\]
such that
\begin{equation}\label{eq:interpolation}
 \ip{w}{v(x_i)}=\ip{w}{v(x_1)}+\ip{g}{x_i-x_1}
 \qquad\text{for }1\leq i\leq k,
\end{equation}
and
\[
 q:=\ip{w}{v(x_0)-\bar v}>d:=\norm{x_0-\bar x},
 \qquad
 \bar x=\sum_{i=1}^k t_i x_i,\qquad
 \bar v=\sum_{i=1}^k t_i v(x_i).
\]
Here $k\leq n\leq\dim Y$ and $\dim L=k-1$.

Note that $\dim w^\perp=\dim Y-1\geq k-1=\dim L$. Hence $w^\perp$ contains a linearly isometric copy of $L$.  We
claim that there is a linear isometric embedding $Q\colon L\to Y$ satisfying
\begin{equation}\label{eq:Q}
 Q^*w=g.
\end{equation}
Here $Q^*\colon Y^*\to L^*$ denotes the adjoint of $Q$.
Indeed, choose an isometric embedding $U\colon L\to w^\perp$, and define $A\colon L\to L$ by
\[
 A=(\mathrm{Id}-gg^*)^{1/2}.
\]
The bound $\norm{g}\leq1$ makes $\mathrm{Id}-gg^*$ positive semidefinite, since for any $z\in L$
\[
 \ip{(\mathrm{Id}-gg^*)z}{z}
 =\norm{z}^2-\ip{g}{z}^2\geq0.
\]
Define 
\[
 Q=wg^*+UA.
\]
Since the range of $U$ is orthogonal to $w$,
\[
 Q^*Q=gg^*+A^2=\mathrm{Id},
 \qquad Q^*w=g,
\]
which proves the claim.

For $R>0$, define
\[
 y_i^{(R)}=v(x_1)-Rw+Q(x_i-x_1)
 \qquad\text{for }1\leq i\leq k.
\]
These points form an isometric copy of $\{x_1,\ldots,x_k\}$.  Let
\[
 b_i=v(x_1)-v(x_i)+Q(x_i-x_1)
 \qquad\text{for }1\leq i\leq k.
\]
Equations \eqref{eq:interpolation} and \eqref{eq:Q} give
\[
 \ip{w}{b_i}
 =\ip{w}{v(x_1)-v(x_i)}+\ip{g}{x_i-x_1}=0
 \qquad\text{for }1\leq i\leq k.
\]
Consequently, if $B=\max\{\norm{b_i}\mid 1\leq i\leq k\}$, then
\begin{equation}\label{eq:rhoR}
 \rho_R:=\max\{\norm{y_i^{(R)}-v(x_i)}\mid 1\leq i\leq k\}
 =\sqrt{R^2+B^2}.
\end{equation}
Writing $\bar y_R=\sum_{i=1}^k t_i y_i^{(R)}$, we see 
\[
 \bar y_R-v(x_0)=-Rw+c,
 \qquad c=v(x_1)-v(x_0)+Q(\bar x-x_1).
\]
Using \eqref{eq:interpolation} and \eqref{eq:Q} once more,
\begin{equation}\label{eq:cprojection}
 \ip{w}{c}
 =\ip{w}{v(x_1)-v(x_0)}+\ip{g}{\bar x-x_1}
 =\ip{w}{\bar v-v(x_0)}=-q.
\end{equation}
Thus $c+qw\perp w$.  Writing $C=\norm{c+qw}$,
Lemma~\ref{lem:isometric-test}, \eqref{eq:rhoR}, and
\eqref{eq:cprojection} yield
\[
 \sqrt{(R+q)^2+C^2}\leq d+\sqrt{R^2+B^2}.
\]
Subtracting the right-hand side from the left-hand side and rationalising
gives
\[
 \begin{aligned}
 0&\geq \sqrt{(R+q)^2+C^2}-d-\sqrt{R^2+B^2}\\
  &=\frac{2Rq+q^2+C^2-B^2}
  {\sqrt{(R+q)^2+C^2}+\sqrt{R^2+B^2}}-d.
 \end{aligned}
\]
The fraction tends to $q$ as $R\to\infty$, so $q-d\leq0$.  This contradicts
$q>d$.  Therefore we have proven that
$(\mathrm E)$ implies$(\mathrm C)$.  The reverse
implication is proven in \cite[Theorem~1.2]{Ciosmak2024}.
\end{proof}

\section{The four-point obstruction}

We next record why 
\cite[Example~4.10]{Ciosmak2024} obstructs the earlier placement argument but not the
characterisation.

\begin{example}[The four-point obstruction]\label{ex:square}
Let
\[
 a_1=(0,0),\quad a_2=(0,1),\quad a_3=(1,0),\quad a_4=(1,1),
\]
let $X=\{a_1,a_2,a_3,a_4\}\subset\R^2$, let $Y=\R^4$, and fix a unit vector
$w\in Y$.  Define
\begin{equation}\label{eq:example-v}
 v(a_1)=v(a_2)=v(a_3)=0,
 \qquad v(a_4)=\frac1{\sqrt2}w.
\end{equation}
This map satisfies $(\mathrm C)$.  To verify this, let $\bar a$ be a convex
combination of the four vertices and let $p$ be the total coefficient of
$a_4$.  The corresponding convex combination of the values of $v$ is
$pw/\sqrt2$.  If $p<1$, write
$\bar a=pa_4+(1-p)r$, where $r\in\conv\{a_1,a_2,a_3\}$.  For $x_0=a_4$,
\[
 \left\|v(a_4)-\frac p{\sqrt2}w\right\|
 =\frac{1-p}{\sqrt2}
 \leq(1-p)\norm{a_4-r}=\norm{a_4-\bar a},
\]
because
$\dist(a_4,\conv\{a_1,a_2,a_3\})=1/\sqrt2$.  For
$x_0=a_1,a_2,a_3$, respectively, taking inner products of $\bar a-a_j$ with
the unit vectors $(1,1)/\sqrt2$, $(1,0)$, and $(0,1)$ gives
\[
 \norm{\bar a-a_j}\geq\frac p{\sqrt2}=
 \norm{v(a_j)-\frac p{\sqrt2}w}.
\]
The case $p=1$ follows directly.

Nevertheless, fix $\delta>0$.  There is no isometric embedding $u\colon X\to Y$ satisfying
\begin{equation}\label{eq:projection-control}
 \left|\ip{u(a)-v(a)}{w}-\delta\right|\leq\varepsilon
 \quad\text{for }a\in X,
 \qquad u(a_1)-v(a_1)=\delta w,
\end{equation}
whenever $0<\varepsilon<1/(5\sqrt2)$.  Indeed, put
$e=u(a_2)-u(a_1)$, $f=u(a_3)-u(a_1)$, and $g=u(a_4)-u(a_1)$.
The six preserved distances give
$\norm{e}=\norm{f}=1$, $\ip{e}{f}=0$, $\norm{g}^2=2$, and
$\ip{g}{e}=\ip{g}{f}=1$; hence $\norm{g-e-f}=0$.  Thus
\[
 u(a_4)=u(a_2)+u(a_3)-u(a_1).
\]
The three constraints at $a_1,a_2,a_3$ imply
$|\ip{u(a_4)-u(a_1)}{w}|\leq4\varepsilon$, whereas the constraint at $a_4$
and the anchor at $a_1$ give
\[
 \left|\ip{u(a_4)-v(a_4)}{w}-\delta\right|
 \geq\frac1{\sqrt2}-4\varepsilon>\varepsilon,
\]
a contradiction.
\end{example}

The four square vertices are affinely dependent.  By contrast,
Proposition~\ref{prop:minimal-simplex} forces the support used in the proof of
Theorem~\ref{thm:main} to be the vertex set of a simplex and supplies an
interpolating vector $g$ with $\norm{g}\leq1$, needed to construct $Q$.
The obstruction therefore cannot occur on the configuration to which the isometric embedding argument is applied.

\section{Applications}

We record two consequences of the new implication: an exact finite branching data processing
characterisation and, under a dimension hypothesis, a lifting theorem for
convex Lipschitz functions that transfers convex Poincar\'e inequalities with
the same constant.

\subsection{Finite branching weak transport}

Weak transport has its origins in entropy--coupling methods for concentration
\cite{Marton1996}.  The framework in which the cost depends nonlinearly on a
conditional law was developed in \cite{GozlanRobertoSamsonTetali2017};
existence, duality, and optimality principles were subsequently established in
\cite{BackhoffBeiglbockPammer2019}.  A recent fundamental theorem establishes
primal attainment and strong duality, characterises optimisers, and, under additional hypotheses, gives dual attainment
\cite{BeiglbockPammerRiessSchrott2025}.  Connections of weak transport with
the Schr\"odinger problem, the Brenier--Strassen theorem, mechanism design,
and semimartingale transport are developed in
\cite{BackhoffVeraguasPammer2022}.

Assume in this subsection that the Hilbert spaces are separable and that
$X\subset Z$ is Borel.  Under $(\mathrm E)$, if $Y\ne\{0\}$,
Theorem~\ref{thm:main} and the case $k=1$ of $(\mathrm C)$ imply that $v$ is
$1$-Lipschitz; when $Y=\{0\}$ this is immediate.  Hence $v$ is Borel
measurable and the pushforwards below are well defined for general probability
measures.  For a real Hilbert space $H$
and $p\geq1$, let $\mathcal P_p(H)$ denote the set of Borel probability measures on
$H$ with finite $p$th moment, and let $W_{p,H}$ denote the $p$-Wasserstein
distance induced by the norm of $H$.  If $E\subset H$ is Borel, we write
$\mathcal P_p(E)$ for the members of $\mathcal P_p(H)$ concentrated on $E$.
The weak transport definitions below are used only when $H$ is separable.

Fix $q\in\{1,2,\ldots\}$.  For $\mu,\nu\in\mathcal P_p(H)$, let
$\Gamma_q(\mu,\nu)$ be the class of all joint triples $(S,T,U)$, defined on a
common probability space, such that $S\sim\mu$, $T\sim\nu$, the auxiliary
random variable $U$ takes values in a standard Borel space, and a regular
conditional kernel $\kappa(s,u;\mathord\cdot)$ for $T$ given $(S,U)$ is
concentrated on at most $q$ points for $\mathcal L(S,U)$-almost every $(s,u)$.
Put
\begin{equation}\label{eq:branching-cost}
 \mathsf B^{[q]}_{p,H}(\nu\mid\mu)
 =\inf\left\{\mathbb E\norm{S-\mathbb E[T\mid S,U]}^p\mid 
              (S,T,U)\in\Gamma_q(\mu,\nu)\right\}.
\end{equation}
Such a kernel exists under the standard Borel assumptions
\cite[Theorem~8.5]{Kallenberg2021}.  We use $\mathcal L(T\mid S,U)$ for the
random probability measure $\kappa(S,U;\mathord\cdot)$; thus, for every Borel
set $B\subset H$, $\mathcal L(T\mid S,U)(B)$ is a version of
$\mathbb P(T\in B\mid S,U)$.  The support condition is independent of the
chosen version up to an $\mathcal L(S,U)$-null set.  Conditional expectations
are understood in the Bochner sense.  For a triple in $\Gamma_q(\mu,\nu)$,
$U$ acts as a branch label: conditionally on $S$, the kernel
$\mathcal L(T\mid S,U)$ gives a mixture decomposition of
$\mathcal L(T\mid S)$ into laws concentrated on at most $q$ points.  No
independence of $U$ from $(S,T)$ is required.  We write
$\Gamma_\infty(\mu,\nu)$ and $\mathsf B^{[\infty]}_{p,H}$ when this restriction on the support
is removed.

For $q=1$, every admissible triple satisfies
$T=\mathbb E[T\mid S,U]$ almost surely, while taking $U=T$ makes every
coupling of $\mu$ and $\nu$ admissible.  Thus
$\mathsf B^{[1]}_{p,H}(\nu\mid\mu)=W_{p,H}^p(\mu,\nu)$.  At the other endpoint,
conditional Jensen's inequality gives
\[
 \mathbb E\norm{S-\mathbb E[T\mid S,U]}^p
 \geq \mathbb E\norm{S-\mathbb E[T\mid S]}^p.
\]
Conversely, restricting to constant $U$ yields the usual barycentric weak
transport cost after taking the infimum over couplings.  Consequently,
\begin{equation}\label{eq:branching-hierarchy}
 \begin{aligned}
 \overline{\mathcal T}_{p,H}(\nu\mid\mu)
 &=\mathsf B^{[\infty]}_{p,H}(\nu\mid\mu)
 \leq\mathsf B^{[q+1]}_{p,H}(\nu\mid\mu)
 \leq\mathsf B^{[q]}_{p,H}(\nu\mid\mu)\\
 &\leq\mathsf B^{[1]}_{p,H}(\nu\mid\mu)=W_{p,H}^p(\mu,\nu),
 \end{aligned}
\end{equation}
where
\[
 \overline{\mathcal T}_{p,H}(\nu\mid\mu)
 =\inf\left\{\mathbb E\norm{S-\mathbb E[T\mid S]}^p\mid 
              S\sim\mu,\ T\sim\nu\right\}
\]
is the usual barycentric weak transport cost
\cite{GozlanRobertoSamsonTetali2017,GozlanJuillet2020}.  In the transformed
coupling used below, the auxiliary variable is $U'=(S,U)$; retaining $S$
prevents atoms from distinct fibres of a noninjective $v$ from being merged
after applying $v$.

Barycentric weak transport is closely connected with convex order.  On the
real line, this connection underlies weak transport--entropy forms of the
convex Poincar\'e inequality
\cite{GozlanRobertoSamsonShuTetali2018}; in finite-dimensional Euclidean space,
for the quadratic cost, it leads to a Brenier--Strassen factorisation of an
optimiser into a deterministic transport followed by a martingale coupling
\cite{GozlanJuillet2020}.  The present
finite branching costs differ from the usual weak costs by imposing a
cardinality bound on the conditional laws almost surely.

If $\mu=\sum_i a_i\delta_{x_i}$ is a finitely supported probability measure on
$X$, then the formula
\[
 v_\#\left(\sum_i a_i\delta_{x_i}\right)
 =\sum_i a_i\delta_{v(x_i)}
\]
defines $v_\#\mu$ without requiring prior measurability of $v$.  General
measures occur below only in the implication from $(\mathrm E)$, where $v$ is
Borel.

\begin{proposition}[Finite branching data processing]
\label{prop:branching-dpi}
Let $1\leq m=\dim Y<\infty$.  Property $(\mathrm E)$ is equivalent to
\begin{equation}\label{eq:branching-dpi}
 \mathsf B^{[m]}_{p,Y}(v_\#\nu\mid v_\#\mu)
 \leq \mathsf B^{[m]}_{p,Z}(\nu\mid\mu)
\end{equation}
for some fixed $p\geq1$ and for all finitely supported probability measures
$\mu,\nu$ on $X$.  If $(\mathrm E)$ holds, then the same inequality holds for
every $p\geq1$, all $\mu,\nu\in\mathcal P_p(X)$, and with $m$ replaced by any
$q$ satisfying $1\leq q\leq m$ on both sides.
\end{proposition}

\begin{proof}
Assume $(\mathrm E)$, fix $p\geq1$ and $1\leq q\leq m$, and let
$\mu,\nu\in\mathcal P_p(X)$.  For
$(S,T,U)\in\Gamma_q(\mu,\nu)$, modify $S$ and $T$ on null sets so that they are
$X$-valued, and set $U'=(S,U)$.  Since
\[
 \sigma(v(S),U')=\sigma(S,U),
\]
$\mathcal L(v(T)\mid v(S),U')$ is the pushforward under $v$ of
$\mathcal L(T\mid S,U)$ and is therefore supported on at most $q$ points.
Thus
\[
 (v(S),v(T),U')\in\Gamma_q(v_\#\mu,v_\#\nu),
 \qquad
 \mathbb E[v(T)\mid v(S),U']=\mathbb E[v(T)\mid S,U].
\]
For $\mathcal L(S,U)$-almost every $(s,u)$, the measure
$\kappa(s,u;\mathord\cdot)$ is concentrated on $X$ and has at most $q$ atoms.
Applying $(\mathrm C)$ pointwise to $\kappa(s,u;\mathord\cdot)$, with
distinguished point $s$, and then evaluating the resulting inequality at
$(S,U)$ gives
\[
 \norm{v(S)-\mathbb E[v(T)\mid S,U]}
 \leq\norm{S-\mathbb E[T\mid S,U]}.
\]
Raising to the $p$th power, taking expectations, and then taking infima proves
\eqref{eq:branching-dpi}.  Conversely, suppose that
\eqref{eq:branching-dpi} holds for some fixed $p\geq1$ and all finitely
supported $\mu,\nu$.  Fix $1\leq k\leq m$, points
$x,x_1,\ldots,x_k\in X$, and coefficients $t_1,\ldots,t_k\geq0$ with
$\sum_{i=1}^k t_i=1$.  Take
$\mu=\delta_x$ and $\nu=\sum_{i=1}^k t_i\delta_{x_i}$.
Since $S=x$ almost surely, conditional Jensen's inequality gives
\[
 \mathbb E\norm{x-\mathbb E[T\mid U]}^p
 \geq\norm{x-\mathbb ET}^p.
\]
The constant choice of $U$ is admissible because $\nu$ has at most $m$ atoms,
and it gives equality.  Applying the same argument after composing with $v$
shows that the two sides of \eqref{eq:branching-dpi} equal, respectively,
\[
 \norm{v(x)-\sum_{i=1}^k t_i v(x_i)}^p
 \quad\text{and}\quad
 \norm{x-\sum_{i=1}^k t_i x_i}^p.
\]
Thus $(\mathrm C)$ holds, and Theorem~\ref{thm:main} applies.
\end{proof}

Thus Proposition~\ref{prop:branching-dpi} identifies $(\mathrm E)$ with the
data processing inequality at level $q=m$ of the hierarchy
\eqref{eq:branching-hierarchy}, whose endpoints are Wasserstein transport and
barycentric weak transport.

\subsection{Convex Lipschitz functions and convex Poincar\'e inequalities}

Talagrand’s convex distance inequality is a foundational concentration result for product probability measures \cite{Talagrand1995}. For symmetric probability measures on the real line, Feldheim, Marsiglietti, Nayar and Wang prove a quantitative equivalence between the convex Poincaré inequality and the convex property \((\tau)\) with a quadratic-linear cost \cite{FeldheimMarsigliettiNayarWang2018}. This property tensorises and yields two-level concentration for convex sets and convex functions of independent real random variables whose laws satisfy it \cite{FeldheimMarsigliettiNayarWang2018}. For arbitrary probability measures on the real line, Gozlan, Roberto, Samson, Shu and Tetali prove that the convex Poincaré inequality is equivalent to the conjunction of two one-sided weak transport-entropy inequalities with a quadratic-linear cost \cite{GozlanRobertoSamsonShuTetali2018}. For probability measures on $\mathbb R^n$, Adamczak and Strzelecki establish the corresponding equivalence with weak transport-entropy inequalities \cite{AdamczakStrzelecki2019}. Shu and Strzelecki use weak transport-entropy inequalities to characterise probability measures on the real line satisfying a class of modified logarithmic Sobolev inequalities for convex functions \cite{ShuStrzelecki2018}. Wintenberger develops weak transport inequalities for non-product laws of weakly dependent time series and applies conditional weak transport to obtain oracle inequalities for least-squares estimators under dependence \cite{Wintenberger2015}.

Write $(\mathrm B)$ for the inequality in $(\mathrm C)$ required for every
positive integer $k$, without the restriction $k\leq\dim Y$.  Suppose that
$(\mathrm E)$ holds.  If $Y$ is
infinite-dimensional, then
Theorem~\ref{thm:main} gives $(\mathrm B)$.

The same conclusion holds when $\dim\operatorname{Aff}X<\infty$ and
$\dim\operatorname{Aff}X<\dim Y$.  To see this, fix $x_0,x_1,\ldots,x_n\in X$, coefficients
$t_1,\ldots,t_n\geq0$ with $\sum_{i=1}^n t_i=1$, and
$b=\sum_{i=1}^n t_i x_i$.  Consider the
fixed barycentre polytope
\[
 P_b=\left\{s\in\mathbb R^n\mid s_i\geq0\text{ for }1\leq i\leq n,\quad
       \sum_{i=1}^n s_i=1,\quad \sum_{i=1}^n s_i x_i=b\right\}.
\]
The support of every extreme point $s$ of $P_b$ is affinely independent:
otherwise there would be a nonzero $a\in\mathbb R^n$, supported on the positive
coordinates of $s$, such that $\sum_i a_i=0$ and $\sum_i a_i x_i=0$; for all
sufficiently small $\varepsilon>0$, the distinct points
$s+\varepsilon a$ and $s-\varepsilon a$ would then belong to $P_b$ and have
midpoint $s$.  Hence the support of every extreme point of $P_b$ has at most $r+1$
points.  Write $t=\sum_{j=1}^l\theta_j s^{(j)}$ as a convex combination of
extreme points $s^{(j)}$, $j=1,\dotsc,l$, of $P_b$.  By Theorem~\ref{thm:main}, property $(\mathrm C)$
holds.  Each $s^{(j)}$, after its zero coordinates have been discarded, has a convex representation
supported on at most $\dim\operatorname{Aff}X+1\leq\dim Y$ points, so $(\mathrm C)$ applies to the
corresponding convex combination; convexity of the norm then gives
\[
 \begin{aligned}
 \norm{v(x_0)-\sum_{i=1}^n t_i v(x_i)}
 &\leq\sum_{j=1}^l\theta_j
   \norm{v(x_0)-\sum_{i=1}^n s_i^{(j)}v(x_i)}\\
 &\leq\sum_{j=1}^l\theta_j\norm{x_0-b}
 =\norm{x_0-b}.
 \end{aligned}
\]
Thus $(\mathrm B)$ holds.

For a function $f\colon H\to\R$ on a Hilbert space $H$, define its
pointwise local Lipschitz constant by
\[
 \operatorname{lip}f(y)
 =\limsup_{\substack{z\to y\\z\ne y}}
   \frac{|f(z)-f(y)|}{\norm{z-y}}.
\]
At an isolated point $y$, we use the convention that
$\operatorname{lip}f(y)=0$.

\begin{proposition}[Lifting convex Lipschitz functions]
\label{prop:convex-lifting}
Assume $X\neq\emptyset$ and $(\mathrm B)$.  For every $L\geq0$, every
finite-valued convex $L$-Lipschitz function
$f\colon Y\to\R$ admits a finite-valued convex $L$-Lipschitz function $F\colon Z\to\R$
such that
\begin{equation}\label{eq:slope-lifting}
 F(x)=f(v(x)),\qquad
 \operatorname{lip}F(x)\leq\operatorname{lip}f(v(x))
 \quad\text{for }x\in X.
\end{equation}
\end{proposition}

The proof, including the pointwise slope estimate, is given at the end of this
section.

\begin{remark}[Affine extension is not automatic]
Example~\ref{ex:square} satisfies $(\mathrm C)$, but there is no affine map
$A\colon \operatorname{Aff}X\to Y$ satisfying $A|_X=v$: indeed,
$a_1+a_4=a_2+a_3$, whereas
$v(a_1)+v(a_4)\ne v(a_2)+v(a_3)$.  Hence the corollary cannot in general be
reduced to a construction $F=f\circ A$ with an affine contraction $A\colon Z\to Y$
extending $v$.
\end{remark}

\begin{corollary}[Convex Poincar\'e inequality transfer without loss]
\label{cor:convex-poincare}
Assume that $X$ is Borel, that $(\mathrm E)$ holds, and either that $Y$ is
infinite-dimensional or that
$\dim\operatorname{Aff}X+1\leq\dim Y$.  Suppose that, for some $C\geq0$,
$\mu\in\mathcal P_2(Z)$ satisfies $\mu(X)=1$ and
\[
 \operatorname{Var}_\mu G
 \leq C\int_Z(\operatorname{lip}G)^2\,d\mu
\]
for every convex Lipschitz $G\colon Z\to\R$.  Then
\[
 \operatorname{Var}_{v_\#\mu}f
 \leq C\int_Y(\operatorname{lip}f)^2\,d(v_\#\mu)
\]
for every convex Lipschitz $f\colon Y\to\R$.
\end{corollary}

\begin{proof}
Property $(\mathrm B)$ holds by the observation preceding
Proposition~\ref{prop:convex-lifting}.  Moreover, the case $k=1$ of
$(\mathrm C)$ implies that $v$ is $1$-Lipschitz, so $\mu\in\mathcal P_2(Z)$ and
$\mu(X)=1$ imply
$v_\#\mu\in\mathcal P_2(Y)$.  Applying the source inequality to the lift $F$
gives
\[
 \operatorname{Var}_{v_\#\mu}f=\operatorname{Var}_\mu F
 \leq C\int_Z(\operatorname{lip}F)^2\,d\mu
\leq C\int_Y(\operatorname{lip}f)^2\,d(v_\#\mu).
\]
Taking the infimum over all admissible constants $C$ shows that the best
constant for $v_\#\mu$ is no larger than that for $\mu$.
\end{proof}

\subsection*{Proof of Proposition~\ref{prop:convex-lifting}}

\begin{proof}
If $Y=\{0\}$, the constant function $F(z)=f(0)$ has all the required
properties.  Henceforth assume that $Y\ne\{0\}$.

For $p\in Y$, define, with the infimum taken over finite convex combinations,
\[
 \begin{aligned}
 H_p(z)=\inf\Biggl\{&
  \sum_{i=1}^n t_i\ip{p}{v(x_i)}
  +\norm p\norm{z-\sum_{i=1}^n t_i x_i}\mid\\[-2pt]
  &n\geq1,\quad x_1,\ldots,x_n\in X,\quad
  t_1,\ldots,t_n\geq0,\quad
  \sum_{i=1}^n t_i=1\Biggr\}.
 \end{aligned}
\]
Condition $(\mathrm B)$ shows that $H_p(x)=\ip{p}{v(x)}$ on $X$.  Fixing
$x_0\in X$, condition $(\mathrm B)$ and the triangle inequality show that the
value corresponding to every admissible finite convex combination is bounded
below by $\ip{p}{v(x_0)}-\norm p\norm{z-x_0}$.  A one-point combination gives
a finite upper bound, so $H_p$ is real-valued.  Taking convex combinations of
two admissible finite convex combinations proves that $H_p$ is convex.  Since
the functions of $z$ appearing in the infimum are all $\norm p$-Lipschitz,
their real-valued infimum $H_p$ is also $\norm p$-Lipschitz.

Let $\mathcal A_f$ be the family of continuous affine minorants of $f$:
\[
 \mathcal A_f=
 \left\{(\alpha,p)\in\R\times Y\mid
 \alpha+\ip{p}{y}\leq f(y)\ \text{for every }y\in Y\right\}.
\]
Since $f$ is finite-valued, convex, and continuous on $Y$, the Hahn--Banach
supporting hyperplane theorem gives, for every $y\in Y$, an element
$(\alpha,p)\in\mathcal A_f$ satisfying
$\alpha+\ip{p}{y}=f(y)$.  Consequently, for every $y\in Y$,
\begin{equation}\label{eq:affine-minorants}
 f(y)=\sup\{\alpha+\ip{p}{y}\mid(\alpha,p)\in\mathcal A_f\}.
\end{equation}
Every $(\alpha,p)\in\mathcal A_f$ satisfies $\norm p\leq L$.  This is
immediate if $p=0$; otherwise, fix $y_0\in Y$, evaluate the minorant along
$y_0+t p/\norm p$, and let $t\to\infty$.

Define
\[
 F(z)=\sup\{\alpha+H_p(z)\mid(\alpha,p)\in\mathcal A_f\}.
\]
For $x\in X$, the identity for $H_p$ and \eqref{eq:affine-minorants} give
$F(x)=f(v(x))$.  In particular, $F$ is finite at $x_*$.  Every function in
the supremum defining $F$ is $L$-Lipschitz, so taking suprema shows that $F$
is finite-valued, convex, and $L$-Lipschitz on $Z$.

For $y\in Y$ and $(\alpha,p)\in\mathcal A_f$, put
\[
 \delta_y(\alpha,p)=f(y)-\alpha-\ip{p}{y}\geq0.
\]
For $\varepsilon>0$, define
\[
 M_y(\varepsilon)
 =\sup\{\norm p\mid (\alpha,p)\in\mathcal A_f,
                    \ \delta_y(\alpha,p)\leq\varepsilon\}.
\]
The defining set is nonempty, $M_y(\varepsilon)\leq L$, and $M_y$ is
nondecreasing.  Let $(\alpha,p)$ be any element of the set defining this
supremum.  If $p\ne0$, put
$h=\sqrt\varepsilon\,p/\norm p$.  The minorant inequality at $y+h$ and
$\alpha+\ip{p}{y}=f(y)-\delta_y(\alpha,p)$ give
\[
 \begin{aligned}
  f(y+h)
  &\geq f(y)-\delta_y(\alpha,p)+\sqrt\varepsilon\norm p,\\
  \norm p
  &\leq
  \frac{f(y+h)-f(y)}{\sqrt\varepsilon}
  +\frac{\delta_y(\alpha,p)}{\sqrt\varepsilon}\\
  &\leq
  \sup\left\{
  \frac{|f(y+u)-f(y)|}{\norm u}\mid
  0<\norm u\leq\sqrt\varepsilon\right\}+\sqrt\varepsilon.
 \end{aligned}
\]
The same bound is trivial when $p=0$.  Taking the supremum over all such
minorants and using the monotonicity of $M_y$ while letting
$\varepsilon\downarrow0$ gives
\begin{equation}\label{eq:minorant-slope}
 \lim_{\varepsilon\downarrow0}M_y(\varepsilon)
 \leq\operatorname{lip}f(y).
\end{equation}

Now fix $x\in X$ and let $y=v(x)$.  Choose
$(\alpha_0,p_0)\in\mathcal A_f$ satisfying
\[
 \alpha_0+\ip{p_0}{y}=f(y),
 \qquad \delta_y(\alpha_0,p_0)=0.
\]
If $p_0\ne0$, then the minorant inequality at
$y+s p_0/\norm{p_0}$ gives, for every $s>0$,
\[
 f(y+s p_0/\norm{p_0})
 \geq f(y)+s\norm{p_0}.
\]
Dividing by $s$ and letting $s\downarrow0$ yields
\begin{equation}\label{eq:touching-slope}
 \norm{p_0}\leq\operatorname{lip}f(y),
\end{equation}
with the case $p_0=0$ immediate.

If $L=0$, then $F$ is $0$-Lipschitz and the desired slope estimate is
immediate.  Assume henceforth that $L>0$.  Let $z\in Z\setminus\{x\}$ and put
$d=\norm{z-x}$.  Since $F(x)=f(y)$, $H_p(x)=\ip{p}{y}$, and
$H_p(z)\leq H_p(x)+\norm p\,d$, we have
\[
\begin{aligned}
 F(z)-F(x)
 &=\sup\{\alpha+H_p(z)-f(y)\mid
          (\alpha,p)\in\mathcal A_f\}\\
 &\leq\sup\{\alpha+\ip{p}{y}+\norm p\,d-f(y)\mid
             (\alpha,p)\in\mathcal A_f\}\\
 &=\sup\{\norm p\,d-\delta_y(\alpha,p)\mid
          (\alpha,p)\in\mathcal A_f\}.
\end{aligned}
\]
Because $\norm p\leq L$, terms with $\delta_y(\alpha,p)>2Ld$ are negative,
whereas the chosen minorant $(\alpha_0,p_0)$ contributes the nonnegative term
$\norm{p_0}d$.  Consequently, restricting the supremum to the set
$\delta_y(\alpha,p)\leq2Ld$ leaves it unchanged.  For the remaining terms,
$\norm p\,d-\delta_y(\alpha,p)\leq\norm p\,d$, and hence
\[
 F(z)-F(x)
 \leq d M_y(2Ld).
\]
Thus
\begin{equation}\label{eq:upper-difference-quotient}
 \frac{F(z)-F(x)}{\norm{z-x}}
 \leq M_y(2L\norm{z-x}).
\end{equation}

To bound the opposite difference, use the $\norm{p_0}$-Lipschitz property of
$H_{p_0}$:
\[
 F(z)\geq\alpha_0+H_{p_0}(z)
 \geq F(x)-\norm{p_0}d.
\]
Consequently,
\begin{equation}\label{eq:lower-difference-quotient}
 \frac{F(x)-F(z)}{\norm{z-x}}
 \leq\norm{p_0}\leq\operatorname{lip}f(y).
\end{equation}
Combining \eqref{eq:upper-difference-quotient} and
\eqref{eq:lower-difference-quotient}, and then using
\eqref{eq:minorant-slope}, yields
\[
 \begin{aligned}
 \operatorname{lip}F(x)
 &=\limsup_{\substack{z\to x\\z\ne x}}
   \frac{|F(z)-F(x)|}{\norm{z-x}}\\
 &\leq\max\left\{
       \lim_{\varepsilon\downarrow0}M_y(\varepsilon),\norm{p_0}\right\}
 \leq\operatorname{lip}f(y).
 \end{aligned}
\]
This proves \eqref{eq:slope-lifting}.
\end{proof}

\section*{Data availability}

No data were used for the research described in the article.

\section*{Declaration of generative AI and AI-assisted technologies in the
manuscript preparation process}

During the preparation of this work, the author used OpenAI Codex to assist in
developing and checking mathematical arguments, organising the manuscript and
literature, editing the language, and preparing the LaTeX source.  The author
subsequently reviewed, verified, and edited the content and takes full
responsibility for the article.


\begin{thebibliography}{99}

\bibitem{AdamczakStrzelecki2019}
R.~Adamczak, M.~Strzelecki,
\textit{On the convex Poincar\'e inequality and weak transportation
inequalities},
Bernoulli 25 (1) (2019) 341--374.
\href{https://doi.org/10.3150/17-BEJ989}{doi:10.3150/17-BEJ989}.

\bibitem{AzagraLeGruyerMudarra2021}
D.~Azagra, E.~Le~Gruyer, C.~Mudarra,
\textit{Kirszbraun's theorem via an explicit formula},
Canad. Math. Bull. 64 (1) (2021) 142--153.
\href{https://doi.org/10.4153/S0008439520000314}
{doi:10.4153/S0008439520000314}.

\bibitem{BackhoffBeiglbockPammer2019}
J.~Backhoff-Veraguas, M.~Beiglb\"ock, G.~Pammer,
\textit{Existence, duality, and cyclical monotonicity for weak transport
costs},
Calc. Var. Partial Differential Equations 58 (2019), Paper No.~203, 28 pp.
\href{https://doi.org/10.1007/s00526-019-1624-y}
{doi:10.1007/s00526-019-1624-y}.

\bibitem{BackhoffVeraguasPammer2022}
J.~Backhoff-Veraguas, G.~Pammer,
\textit{Applications of weak transport theory},
Bernoulli 28 (1) (2022) 370--394.
\href{https://doi.org/10.3150/21-BEJ1346}
{doi:10.3150/21-BEJ1346}.

\bibitem{BeiglbockPammerRiessSchrott2025}
M.~Beiglb\"ock, G.~Pammer, L.~Riess, S.~Schrott,
\textit{The fundamental theorem of weak optimal transport},
preprint (2025).
\href{https://arxiv.org/abs/2501.16316}{arXiv:2501.16316}.

\bibitem{CavagnariSavareSodini2025}
G.~Cavagnari, G.~Savar\'e, G.E.~Sodini,
\textit{Extension of monotone operators and Lipschitz maps invariant for a
group of isometries},
Canad. J. Math. 77 (1) (2025) 149--186.
\href{https://doi.org/10.4153/S0008414X23000846}
{doi:10.4153/S0008414X23000846}.

\bibitem{CiosmakThesis}
K.J.~Ciosmak,
\textit{Optimal Transport and 1-Lipschitz Maps},
DPhil thesis, University of Oxford, 2020.
\href{https://doi.org/10.5287/ora-nompqy5bn}{doi:10.5287/ora-nompqy5bn}.

\bibitem{Ciosmak2021}
K.J.~Ciosmak,
\textit{Continuity of extensions of Lipschitz maps},
Israel J. Math. 245 (2) (2021) 567--588.
\href{https://doi.org/10.1007/s11856-021-2215-0}{doi:10.1007/s11856-021-2215-0}.

\bibitem{Ciosmak2024}
K.J.~Ciosmak,
\textit{Continuity of extensions of Lipschitz maps and of monotone maps},
J. Lond. Math. Soc. (2) 110 (5) (2024) e70014.
\href{https://doi.org/10.1112/jlms.70014}{doi:10.1112/jlms.70014}.

\bibitem{EspinolaNicolae2015}
R.~Esp\'{\i}nola, A.~Nicolae,
\textit{Continuous selections of Lipschitz extensions in metric spaces},
Rev. Mat. Complut. 28 (3) (2015) 741--759.
\href{https://doi.org/10.1007/s13163-015-0171-0}
{doi:10.1007/s13163-015-0171-0}.

\bibitem{FeldheimMarsigliettiNayarWang2018}
N.~Feldheim, A.~Marsiglietti, P.~Nayar, J.~Wang,
\textit{A note on the convex infimum convolution inequality},
Bernoulli 24 (1) (2018) 257--270.
\href{https://doi.org/10.3150/16-BEJ875}
{doi:10.3150/16-BEJ875}.

\bibitem{GozlanJuillet2020}
N.~Gozlan, N.~Juillet,
\textit{On a mixture of Brenier and Strassen theorems},
Proc. Lond. Math. Soc. (3) 120 (3) (2020) 434--463.
\href{https://doi.org/10.1112/plms.12302}{doi:10.1112/plms.12302}.

\bibitem{GozlanRobertoSamsonShuTetali2018}
N.~Gozlan, C.~Roberto, P.-M.~Samson, Y.~Shu, P.~Tetali,
\textit{Characterization of a class of weak transport-entropy inequalities
on the line},
Ann. Inst. Henri Poincar\'e Probab. Stat. 54 (3) (2018) 1667--1693.
\href{https://doi.org/10.1214/17-AIHP851}
{doi:10.1214/17-AIHP851}.

\bibitem{GozlanRobertoSamsonTetali2017}
N.~Gozlan, C.~Roberto, P.-M.~Samson, P.~Tetali,
\textit{Kantorovich duality for general transport costs and applications},
J. Funct. Anal. 273 (11) (2017) 3327--3405.
\href{https://doi.org/10.1016/j.jfa.2017.08.015}
{doi:10.1016/j.jfa.2017.08.015}.

\bibitem{Kallenberg2021}
O.~Kallenberg,
\textit{Foundations of Modern Probability},
3rd ed., Probability Theory and Stochastic Modelling, vol.~99,
Springer, Cham, 2021.
\href{https://doi.org/10.1007/978-3-030-61871-1}
{doi:10.1007/978-3-030-61871-1}.

\bibitem{Kirszbraun1934}
M.~Kirszbraun,
\textit{\"Uber die zusammenziehende und Lipschitzsche Transformationen},
Fund. Math. 22 (1934) 77--108.
\href{https://doi.org/10.4064/fm-22-1-77-108}{doi:10.4064/fm-22-1-77-108}.

\bibitem{Kopecka2012Bootstrapping}
E.~Kopeck\'a,
\textit{Bootstrapping Kirszbraun's extension theorem},
Fund. Math. 217 (1) (2012) 13--19.
\href{https://doi.org/10.4064/fm217-1-2}{doi:10.4064/fm217-1-2}.

\bibitem{Kopecka2012Continuous}
E.~Kopeck\'a,
\textit{Extending Lipschitz mappings continuously},
J. Appl. Anal. 18 (2) (2012) 167--177.
\href{https://doi.org/10.1515/jaa-2012-0011}
{doi:10.1515/jaa-2012-0011}.

\bibitem{KopeckaReich2011}
E.~Kopeck\'a, S.~Reich,
\textit{Continuous extension operators and convexity},
Nonlinear Anal. 74 (18) (2011) 6907--6910.
\href{https://doi.org/10.1016/j.na.2011.07.013}
{doi:10.1016/j.na.2011.07.013}.

\bibitem{LangSchroeder1997}
U.~Lang, V.~Schroeder,
\textit{Kirszbraun's theorem and metric spaces of bounded curvature},
Geom. Funct. Anal. 7 (3) (1997) 535--560.
\href{https://doi.org/10.1007/s000390050018}
{doi:10.1007/s000390050018}.

\bibitem{Marton1996}
K.~Marton,
\textit{Bounding $\bar d$-distance by informational divergence: a method to
prove measure concentration},
Ann. Probab. 24 (2) (1996) 857--866.
\href{https://doi.org/10.1214/aop/1039639365}
{doi:10.1214/aop/1039639365}.

\bibitem{ReichSimons2005}
S.~Reich, S.~Simons,
\textit{Fenchel duality, Fitzpatrick functions and the Kirszbraun--Valentine
extension theorem},
Proc. Amer. Math. Soc. 133 (9) (2005) 2657--2660.
\href{https://doi.org/10.1090/S0002-9939-05-07983-9}
{doi:10.1090/S0002-9939-05-07983-9}.

\bibitem{ShuStrzelecki2018}
Y.~Shu, M.~Strzelecki,
\textit{A characterization of a class of convex log-Sobolev inequalities on
the real line},
Ann. Inst. Henri Poincar\'e Probab. Stat. 54 (4) (2018) 2075--2091.
\href{https://doi.org/10.1214/17-AIHP865}
{doi:10.1214/17-AIHP865}.

\bibitem{Talagrand1995}
M.~Talagrand,
\textit{Concentration of measure and isoperimetric inequalities in product
spaces},
Publ. Math. Inst. Hautes \`Etudes Sci. 81 (1995) 73--205.
\href{https://doi.org/10.1007/BF02699376}
{doi:10.1007/BF02699376}.

\bibitem{Valentine1945}
F.A.~Valentine,
\textit{A Lipschitz condition preserving extension for a vector function},
Amer. J. Math. 67 (1945) 83--93.
\href{https://doi.org/10.2307/2371917}{doi:10.2307/2371917}.

\bibitem{Wintenberger2015}
O.~Wintenberger,
\textit{Weak transport inequalities and applications to exponential and
oracle inequalities},
Electron. J. Probab. 20 (2015), Paper No.~114, 27 pp.
\href{https://doi.org/10.1214/EJP.v20-3558}
{doi:10.1214/EJP.v20-3558}.

\end{thebibliography}
\end{document}